\documentclass[a4paper,12pt]{article}
\usepackage{graphicx}

\usepackage[margin=1 in]{geometry}
\usepackage[hidelinks]{hyperref}
\renewcommand{\url}[1]{{\href{#1}{#1}}}

\graphicspath{{NC18GTimages/}}

 \usepackage{times} 

\usepackage{mathtools}

\usepackage{amsfonts, amsmath, amsthm, amssymb} 

\title{A Central Limit Theorem and Exponential Correlation Decay for the Coulomb Chain}
\author{Henrik Ekström\footnote{This work was partially supported by the Wallenberg AI, Autonomous Systems and Software Program (WASP) funded by the Knut and Alice Wallenberg Foundation.}}

\date{}

\theoremstyle{plain}
\newtheorem{theorem}    {\bf Theorem}  [section]

\newtheorem{lemma}        [theorem]  {\bf Lemma}
\newtheorem{definition}   [theorem]  {\bf Definition}

\newtheoremstyle{nodotversion}{}{}{}{}{}{}{0pt}{}

\theoremstyle{nodotversion}
\newtheorem*{theorem*}  {\bf Theorem}

\theoremstyle{definition}

\numberwithin{equation}{section}

\newcommand{\posintegers}   {\mathbb{N}}

\newcommand{\real}          {\mathbb{R}}
\newcommand{\prob}          {\mathbb{P}}
\newcommand{\E}             {\mathbb{E}}

\newcommand{\T}             {\mathcal{T}}
\renewcommand{\O}             {\mathcal{O}}

\newcommand{\halfquad}      {\:\;}

\renewcommand{\vec}[1]  {\bar{#1}}

\renewcommand{\epsilon}{\varepsilon }

\begin{document}
\maketitle 
\smallskip
\noindent 
\small{Mathematical Center, Faculty of Science, University of Lund, Sölvegatan 18, 22100, Lund, Sweden.}
\smallskip
\begin{abstract}
    We study the Coulomb chain where particles are restricted to one dimension and experience three-dimensional Coulomb interactions with their nearest and next-to-nearest neighbours. The distances between consecutive particles are treated as random variables. It is shown that the correlation between clusters of consecutive variables decay exponentially with the number of variables separating them. This result is then used to prove a central limit theorem. 
\end{abstract}

\section{Introduction}
The Coulomb gas is a widely studied object in statistical physics, see e.g. the recent survey \cite{Djalil21-AspctOfCG}. We focus on a particular case of a Coulomb ensemble where particles with pairwise three-dimensional Coulomb interactions are confined to one dimension (the confinement could e.g. be achieved by a strong external force \cite{Dubin97-MinEnof1DCoulCh}).  Recall that the strength of the Coulomb force between two particles  is reciprocal to the distance between them.
The Coulomb potential can also be considered as a special case of the long-range Riesz potential \cite{Lewin22RieszKnandUnkn}. Models of this type exhibit rich behaviour involving phase transitions and they appear not only in the field of statistical physics but also in e.g. random matrix theory.

Consider $N$ particles of equal charge on a line.
Let $\vec P\in\real^{N}$ denote the consecutive positions of these particles, so that we have
\begin{equation}
    0\leq P_1 
    \leq \cdots 
    \leq P_{N-1} 
    \leq P_{N}.
\end{equation}
(One may fix the position of the left-most particle at zero.)
Let $K\in\posintegers$ be the number of closest neighbours (in either direction) that each particle interacts with. For $k\in\{1,\dots, K\}$, let $\beta_k\in\real^+$, be a parameter governing the strength of that level of interaction. The potential energy of the configuration $\vec{P}$ of particles is given by the sum of pairwise interactions between them:
\begin{equation}
    H(\vec{P}) = \sum_{k=1}^K\sum_{j=1}^{N-k} \frac{\beta_k}{P_{j+k}-P_{j}}.
\end{equation}

Asymptotically, the influence from the boundaries on the particles in the middle should decrease (at least for the case $K=2$
this was proved in \cite{Turova_22}).
To avoid boundary effect we shall henceforth set the particles on the circle, treating particle $N$ as a neighbour to particle $1$.
Setting also 
$\vec{Y}\in\real^N$ to be the vector of distances between the particles,
\begin{equation}\begin{aligned}
    Y_1 &\coloneqq P_1,\\
    Y_i &\coloneqq P_{i}-P_{i-1},\ i\in\{2,\dots,N\},
\end{aligned}\end{equation}
we define a circular form 
of the energy function in terms of $\vec Y$:
\begin{equation}
    H^{\circ}(\vec{Y})= \sum_{k=1}^K
    \sum_{i=1}^{N} 
        \frac{\beta_k}{Y_i+\dots+Y_{i+k-1}}.
\end{equation}
(For ease of notation, all indices here and henceforth are understood to be taken modulo $N$.)

We shall study the properties of the random 
configurations of particles given by the following 
Gibbs distribution density:
\begin{equation}
    f_{\vec Y}(\vec y) \coloneqq \frac{1}{Z_N} e^{-H^\circ(\vec y)}, \ \ \ \vec y\in[0,1]^N,
    \label{eq:YdistDefinition}
\end{equation}
where $Z_N$ is a normalising constant given by
\begin{equation}
    Z_N \coloneqq \int_{[0,1]^N}e^{-H^\circ(\vec y)}d\vec y.
\end{equation}
We assume here that the distance between consecutive particles is bounded, and therefore without loss of generality  we let $Y_i\in[0,1]$ in (\ref{eq:YdistDefinition}). 
Observe, however, that the length of the one dimensional segment where the particles are placed is not fixed, it is
\begin{equation}
    0\leq  \sum_{i=1}^N Y_i \leq N.
\end{equation}
Note that the distribution of $Y_i$ is a function of $N$, and to underline this dependence when needed, we shall write
\begin{equation}
    Y_i = Y_{i,N}.
\end{equation}

The model under consideration was initially introduced by Malyshev in 2014 \cite{Malyshev14-PhaseTI} and has been further studied since then \cite{MalyshevZamyatin15-CMultSysts, Turova16PhaseTin1DCG, Turova_22, Turova23-AntiferromagPropOfCC}. It differs from the more studied cases in that the particles do not interact with all other particles, but only with a fixed number of nearest neighbours. 
Following the initial work  \cite{Malyshev14-PhaseTI}, 
several phase transitions is proven to occur 
in this model in the presence of an increasing external force but only for the case $K=1$ \cite{Turova16PhaseTin1DCG}.
The case $K\geq 2$ is intrinsically different, as the inter-particle distances become dependent unlike in the case $K=1$.
When $K=2$, the exponential decay of correlations between inter-spacings is proven in \cite{Turova_22}, but a central limit theorem is only conjectured. 
Recently the covariance structure in the case $K=2$ has been described, as well as a central limit theorem proved but under the assumption that all the particles are confined to an interval of a fixed length \cite{Turova23-AntiferromagPropOfCC}.

The aim here is to prove the central limit theorem for particle inter-spacings when the length is not fixed. It should be noted that the related results \cite{Boursier23-OptLocLawsandCLT} on the central limit theorem for long-range Riesz gas ensembles concern the model with $K=N$, i.e., when all pairwise interactions are taken into account. Observe, however, that the case of exactly the Coulomb potential is beyond the scope of that work.

\section{Results}
We establish the central limit theorem for the entries of $\vec Y$ defined by Equations \eqref{eq:YdistDefinition} when $K=2$,
conjectured (without much details) in \cite{Turova_22}.
In what follows, we will write the parameters as $\beta_1=\beta$ and $\beta_2=\gamma$,
consistent with the notation in \cite{Turova_22}.
Although the most physically justified case is perhaps $\beta=\gamma$, we keep these parameters separate to underline their different roles. 
Our method is applicable for the more general case when particles interact with their $K\geq 2$ nearest neighbours as well, with added technical details. The most important feature of our model is the dependence of the random variables $Y_i$ expressed though the Gibbs form of the pair-wise interactions in the system. The correlation decay between any two subsets of consecutive particles established below is a crucial ingredient in the proof of the central limit theorem, as well as an informative result of the model  on its own.

\subsection{Correlation Decay in the Coulomb Gas Ensemble}
\label{subsection:CoulombResults}

For any non-empty disjoint subsets of consecutive indices $I_{},J_{}\subset\{1,\dots,N\}$, let 
\begin{equation}
    r(I_{},J_{})
    \label{eq:rDef}
\end{equation}
be the distance between $I_{}$ and $J_{}$, which is defined as the smallest number of the consecutive indices between these sets; in other words, it is the
cardinal of the smaller set of consecutive indices in 
\begin{equation}
    \{1,\dots,N\}\setminus (I_{}\cup J_{}).
\end{equation}
Here the indices $1$ and $N$ are considered as neighbours, so $I$, $J$ and the two (potentially empty) sets of indices `between' them are in general allowed to wrap around from $N$ to $1$. For any subset $J\subset\{1,\dots,N\}$, we use the notation
\begin{equation}
    \vec{Y}_{J} \coloneqq (Y_j)_{j\in J}.
\end{equation}

\begin{theorem} \label{thm:clusterDecayCoulomb}
    Let $\beta>0$ and $\gamma\geq 0$ be arbitrary parameters and let $I,J\subset\{0,\dots,N\}$ be disjoint sets of consecutive indices. Let $\vec{Y}\in[0,1]^N$ be a random vector with density given by
    \begin{equation}\begin{aligned}
        f_{\vec Y}(\vec y) &= 
            \frac{1}{Z_N}
            e^{-\sum_{i=1}^{N} \left(\frac{\beta}{2y_i}+\frac{\gamma}{y_i+y_{i+1}}\right)},\\
            Z_N &= \int_{[0,1]^N}e^{-\sum_{i=1}^{N} \left(\frac{\beta}{2y_i}+\frac{\gamma}{y_i+y_{i+1}}\right)}d\vec y.
        \label{eq:YdistNextNearest}
    \end{aligned}\end{equation}
    Then there exist positive constants $C=C(\beta,\gamma)$ and $\alpha=\alpha(\gamma)$ such that
    \begin{equation}\begin{aligned}
        \Big| f_{\vec{Y}_{I_{}}|\vec{Y}_{J_{}}=\vec{y}_{J_{}}}(\vec{y}_{I_{}}) - f_{\vec{Y}_{I_{}}}(\vec{y}_{I_{}}) \Big|
            \leq
        Ce^{-\alpha r(I_{},J_{})} f_{\vec{Y}_{I_{}}}(\vec{y}_{I_{}}),
        \label{eq:clusterDecayBound}
    \end{aligned}\end{equation}
    for any $\vec{y}_{I_{}} \in[0,1]^{|I_{}|}$ and $\vec{y}_{J_{}} \in[0,1]^{|J_{}|}$. Furthermore, we have that
    \begin{equation}
        \alpha(\gamma)\overset{\gamma \to0}{\longrightarrow} \infty.
        \label{eq:alphaToInfty}
    \end{equation}
\end{theorem}
This theorem extends the result of Theorem 2.1 in \cite{Turova_22}, which only treats the case when $|I|=|J|=1$, albeit also in the non-circular version of the model. Note that if $\gamma=0$ the density in Equation \eqref{eq:YdistNextNearest} factorises and $Y_i$ are independent, which agrees with the behaviour of $\alpha$ in Equation \eqref{eq:alphaToInfty}.

\subsection{The Central Limit Theorem}
\label{subsection:CLTResults}

Our main result here is the following central limit theorem. 

\begin{theorem}
    Let $\beta>0$ and $\gamma\geq 0$ be arbitrary parameters and $\vec{Y}\in[0,1]^N$ be a random vector with density given by Equation \eqref{eq:YdistNextNearest}.
    Denote
    \begin{equation}
        \sigma_N^2 \coloneqq \frac{1}{N}\text{Var}\bigg(\sum_{i=1}^N \big(Y_i-\E(Y_i)\big)\bigg),
        \label{eq:sigma}
    \end{equation}
    and set
    \begin{equation}
        \zeta_N\coloneqq \frac{1}{\sqrt{\smash[b]{N\sigma^2_N}}}\sum_{i=1}^N \big(Y_i-\E(Y_i)\big).
        \label{thmeq:sumAsZeta}
    \end{equation}
    Then for any $\epsilon\in(0,1/4)$, 
    we have
    \begin{equation}\begin{aligned}
        \sup_{z}\left| \prob(\zeta_N \leq z)-\prob(Z \leq z)\right| 
            =
        \O(N^{-\frac{1}{4}+\epsilon}),
        \label{thmeq:distributions}
    \end{aligned}\end{equation}
    where $Z$ is a standard normal random variable.
    \label{thm:CLT}
\end{theorem}

The proof of Theorem \ref{thm:CLT} is 
much inspired by the work of Schmuland and Sun \cite{SS_2004_CLTforRFldWExpDecay} (see also the discussion and the reference therein on similar mixing conditions, particularly in the context of statistical physics). Note, however, that the results  of \cite{SS_2004_CLTforRFldWExpDecay} are not directly applicable for our setting. In particular, 
the random variables in \cite{SS_2004_CLTforRFldWExpDecay} are 
placed on (or indexed by) an infinite lattice. Moreover, the conditions which are assumed in
\cite{SS_2004_CLTforRFldWExpDecay},  remain to be established for our model.

Taking advantage of some specific features (e.g. the smooth density), 
we are able to establish a better rate of convergence for our model, which is arbitrarily close to $\O(N^{-1/4})$, while the rate provided by \cite{SS_2004_CLTforRFldWExpDecay}
 is $\O(N^{-1/9})$.
We conjecture that Equation (\ref{thmeq:distributions}) provides the best possible rate for the central limit theorem for this model. At least we show that the rate can not be improved further with the methods used here.
The methods should also be useful for studying more general cases of Gibbs distributions.

\section{Proofs}
\subsection{Proof of Theorem \ref{thm:clusterDecayCoulomb}}
\label{subsection:ClusterDecay}
Let $I$ and $J$ be disjoint sets of consecutive indices,
\begin{equation}
    I,J\subset\{1,\dots,N\}.
\end{equation}
Since the model is circular and $Y_i$ identically distributed, we shall without loss of generality assume that 
\begin{equation}\begin{aligned}
    I &= \{1,\dots,|I|\},
    \label{eq:IstartsAtOne}
\end{aligned}\end{equation}
and also that $1\leq|I|\leq |J|$.

Theorem 2.1 in \cite{Turova_22} proves a similar result to our theorem. There, the distribution is not circular and, using our notation, $|I_{}|=|J_{}|=1$. The proof of our theorem follows the proof in \cite{Turova_22} but it needs some adjustments to account for the clusters of variables. However, the case $|I|=|J|=1$ can be omitted below as it directly follows from \cite{Turova_22}. 

We begin by defining some useful functions and derive bounds for them. By defining the function
\begin{equation}\begin{aligned}
    Q(x,y) &\coloneqq e^{-\frac{\beta}{2x}-\frac{\beta}{2y}-\frac{\gamma}{x+y}},
    \label{eq:Qdef}
\end{aligned}\end{equation}
we may rewrite the density in Equation \eqref{eq:YdistNextNearest} as
\begin{equation}\begin{aligned}
    f_{\vec Y}(\vec y) {}=&\ \frac{
        1
    }{
        Z_N
    }
        \prod_{i=1}^{N}Q(y_i,y_{i+1}),\\
    Z_N {}\coloneqq& \int_{[0,1]^N}
        \prod_{i=1}^{N}Q(y_i,y_{i+1}) d\vec y.
        \label{eq:YdistNextNearestWithQ}
\end{aligned}\end{equation}
Further, for $r\in\{0,1,\dots,N-1\}$ and $x,y\in[0,1]$, we define the function 
\begin{equation}\begin{aligned}
    \T^r(x,y) &\coloneqq 
        \int_{[0,1]^r} Q(x,u_1)Q(u_1,u_2)\cdots Q(u_{r-1},u_r)Q(u_r,y) du_1\cdots du_r,
\end{aligned}\end{equation}
and note that
\begin{equation}\begin{aligned}
    \T^0(x,y) = Q(x,y).
\end{aligned}\end{equation}
We now derive approximations of this function $\T^r(x,y)$ for use in what follows. By using the upper bound
\begin{equation}\begin{aligned}
    Q(x,y)
        &\leq e^{-\frac{\beta}{2x}}e^{-\frac{\beta}{2y}},
\end{aligned}\end{equation}
we may, for any $r\in\posintegers\setminus\{0,1\}$, write 
\begin{equation}\begin{aligned}
    \T^r(x,y) &= \int_{[0,1]}Q(x,x')\T^{r-1}(x',y)dx'\\
    &\leq e^{-\frac{\beta}{2x}}
        \int_{[0,1]}
            e^{-\frac{\beta}{2x'}}
            \T^{r-1}(x',y)
        dx'\\
    &\leq e^{-\frac{\beta}{2x}}
        \left(\int_{[0,1]^2}
            e^{-\frac{\beta}{2x'}}
            \T^{r-2}(x',y')
            e^{-\frac{\beta}{2y'}}
        dx'dy'\right)
        e^{-\frac{\beta}{2y}},
    \label{eq:TapproximationsRgeqTwo}
\end{aligned}\end{equation}
and by using the lower bound
\begin{equation}\begin{aligned}
    Q(x,y) &\geq e^{-\frac{\beta}{2x}}e^{-\frac{\beta}{2y}-\frac{\gamma}{y}},
\end{aligned}\end{equation}
we may similarly, for any $r\in\posintegers\setminus\{0,1\}$, write 
\begin{equation}\begin{aligned}
    \T^r(x,y)
    &\geq e^{-\frac{\beta}{2x}}
        \left(
            \int_{[0,1]}
                e^{-\frac{\beta+2\gamma}{2x'}}
                \T^{r-1}(x',y)
            dx'
        \right)\\
    &\geq e^{-\frac{\beta}{2x}}
        \left(
            \int_{[0,1]}
                e^{-\frac{\beta+2\gamma}{x'}}
            dx'
        \right)
        \left(
            \int_{[0,1]}
                e^{-\frac{\beta}{2x''}}
                \T^{r-2}(x'',y)
                e^{-\frac{\beta}{2y}}
            dx''
        \right)\\
        &= ce^{-\frac{\beta}{2x}}
        \left(
            \int_{[0,1]}
                e^{-\frac{\beta}{2x''}}
                \T^{r-2}(x'',y)
                e^{-\frac{\beta}{2y}}
            dx''
        \right),
        \label{eq:TapproximationsRgeqFour}
\end{aligned}\end{equation}
where
\begin{equation}
    c=c(\beta,\gamma)\coloneqq \int_{[0,1]}
                e^{-\frac{\beta+2\gamma}{u}}du.
\end{equation}
Note that $c\in(0,1]$ is a non-zero constant for all values of the parameters $\beta$ and $\gamma$.
The first bound in both Equations \eqref{eq:TapproximationsRgeqTwo} and \eqref{eq:TapproximationsRgeqFour} is valid also for $r=1$.

Further, we write the similar approximations that for any $m,n\in\posintegers$ we have
\begin{equation}\begin{aligned}
    \T^{m+2+n}(u,v)
        &\leq
    \left(\int_{[0,1]}
    \T^{m}(u,u')
    e^{-\frac{\beta}{2u'}}
    du'\right)
    \left(\int_{[0,1]}
    e^{-\frac{\beta}{2v'}}
    \T^{n}(v',v)
    dv'
    \right)
    ,
    \label{eq:splitTinTheMiddle}
\end{aligned}\end{equation}
and
\begin{equation}\begin{aligned}
    \T^{m+3+n}(u,v)
        &\geq
    \left(\int_{[0,1]}
    \T^{m}(u,u')
    e^{-\frac{\beta}{2u'}}
    du'\right)
    c
    \left(\int_{[0,1]}
    e^{-\frac{\beta}{2v'}}
    \T^{n}(v',v)
    dv'
    \right)
    .
    \label{eq:splitTinTheMiddle-Denominator}
\end{aligned}\end{equation}

Let $p_1=|I|$, $p_2=|J|$, and let $r$ and $r'$ be the number of consecutive indices separating $I$ and $J$. Without loss of generality we assume that $r\leq r'$. Hence by the definition in Equation \eqref{eq:rDef},
\begin{equation}
    r(I,J) = r.
\end{equation}
This allows us to explicitly write the index sets as
\begin{equation}\begin{aligned}
    I &= \{1,\dots,p_1\},\\
    J &= \{ p_1+r+1, \dots, p_1+r+p_2\}.
\end{aligned}\end{equation}
Assume, for now, that $p_1,p_2\geq 2$.

Using the functions $Q$ and $\T$, we now write the density of $\vec Y_{I}$, Equation \eqref{eq:YdistNextNearestWithQ}, as
    \begin{equation}\begin{aligned}
    f_{\vec Y_I}(\vec y_I) 
    &=
    \frac{
    1}{Z_N}
        \bigg(\prod_{i=1}^{p_1-1}Q(y_i,y_{i+1})\bigg) \T^{r+p_2+r'}(y_{p_1},y_1),
    \label{eq:densityWithT}
\end{aligned}\end{equation}
and the density of $\vec Y_{I}$ conditioned on $\vec Y_{J}$ as
\begin{equation}\begin{aligned}
    f_{\vec Y_I|\vec Y_J=\vec y_J}(\vec y_I)
    &= \frac{
        \Big(\!\!\prod_{i=1}^{p_1-1}Q(y_i,y_{i+1})\!\Big)  \T^{r}(y_{p_1},y_{p_1+r+1}) \Big(\!\!\prod_{j=p_1+r+1}^{p_1+r+p_2-1}Q(y_j,y_{j+1})\!\Big) \T^{r'}(y_{p_1+r+p_2},y_1)
    }{
        \Big( \prod_{j=p_1+r+1}^{p_1+r+p_2-1}Q(y_j,y_{j+1})\Big) \T^{r'+p_1+r}(y_{p_1+r+p_2},y_{p_1+r+1})
    }\\[5pt]
    &= \frac{
        \Big(\prod_{i=1}^{p_1-1}Q(y_i,y_{i+1})\Big)  \T^{r}(y_{p_1},y_{p_1+r+1})  \T^{r'}(y_{p_1+r+p_2},y_1)
    }{
        \T^{r'+p_1+r}(y_{p_1+r+p_2},y_{p_1+r+1})
    }\\[5pt]
    &= \frac{
        Z_N  \T^{r}(y_{p_1},y_{p_1+r+1})  \T^{r'}(y_{p_1+r+p_2},y_1)
    }{
        \T^{r+p_2+r'}(y_{p_1},y_{y_1})
        \T^{r'+p_1+r}(y_{p_1+r+p_2},y_{p_1+r+1})
    }f_{\vec Y_I}(\vec y_I).
    \label{eq:conditionalDensityWithT}
\end{aligned}\end{equation}

Combining Equations \eqref{eq:densityWithT} and \eqref{eq:conditionalDensityWithT} gives the expression
\begin{equation}\begin{aligned}
    &f_{\vec Y_I|\vec Y_J=\vec y_J}(\vec y_I) - f_{\vec Y_I}(\vec y_I)
    =\left( \frac{
        Z_N\T^{r}(y_{p_1},y_{p_1+r+1})  \T^{r'}(y_{p_1+r+p_2},y_1)
    }{
        \T^{r+p_2+r'}(y_{p_1},y_1)
        \T^{r'+p_1+r}(y_{p_1+r+p_2},y_{p_1+r+1})
    }
    -1\right)f_{\vec Y_I}(\vec y_I),
    \label{eq:expressionToBound}
\end{aligned}\end{equation}
so we require an absolute bound on the bracket on the right rand side, i.e. the function
\begin{equation}\begin{aligned}    g(u_1,u_2,u_3,u_4) &\coloneqq \frac{
        Z_N
        \T^{r}(u_2,u_3) \T^{r'}(u_4,u_1)
        - \T^{r+p_2+r'}(u_2,u_1)
        \T^{r'+p_1+r}(u_4,u_3)
    }{
        \T^{r+p_2+r'}(u_2,u_1)
        \T^{r'+p_1+r}(u_4,u_3)
    },
    \label{eq:clusterEqOne}
\end{aligned}\end{equation}
where the variables $y_1,y_{p_1},y_{p_1+r+1}$ and $y_{p_1+r+p_2}$ have been replaced by $u_1,u_2,u_3$ and $u_4$, respectively, for increased legibility. We now rewrite $g(\vec u)$ by making explicit some of the integrals hidden within the notation.
The dummy variables $v_1,\dots,v_4$ in what follows correspond to the integrated arguments $y_1,y_{p_1},y_{p_1+r+1}$ and $y_{p_1+r+p_2}$, respectively.
We rewrite $Z_N$ as
\begin{equation}\begin{aligned}
    Z_N
    &= \int_{[0,1]^4} \T^{p_1-2}(v_1,v_2)\T^{r}(v_2,v_3) \T^{p_2-2}(v_3,v_4) \T^{r'}(v_4,v_1) dv_1 dv_2 dv_3 dv_4.
    \label{eq:ZnwithT}
\end{aligned}\end{equation}
We rewrite the second term in the numerator of \eqref{eq:clusterEqOne} by making two integrals in each factor explicit:
\begin{equation}\begin{aligned}
    &\T^{r+p_2+r'}(u_2,u_1)\T^{r'+p_1+r}(u_4,u_3) \\
    &\quad= \int_{[0,1]^2}\T^{r}(u_2,v_3)\T^{p_2-2}(v_3,v_4)\T^{r'}(v_4,u_1) dv_3dv_4\\
           &\quad\quad\times 
    \int_{[0,1]^2}\T^{r'}(u_4,v_1) \T^{p_1-2}(v_1,v_2)  \T^{r}(v_2,u_3) dv_1dv_2.
    \label{eq:TfactorsTwo}
\end{aligned}\end{equation}
Equations \eqref{eq:ZnwithT} and \eqref{eq:TfactorsTwo} inserted into Equation \eqref{eq:clusterEqOne} leads to 
\begin{equation}\begin{aligned}
    g(\vec u)
    &= \int_{[0,1]^4}\frac{
            \T^{p_1-2}(v_1,v_2)
            \T^{p_2-2}(v_3,v_4)
        }{
        \T^{r+p_2+r'}(u_2,u_1)
        \T^{r'+p_1+r}(u_4,u_3)
        }\\
            &\quad
        \times \Big(
            \T^{r}(u_2,u_3)
            \T^{r}(v_2,v_3)
            \T^{r'}(u_4,u_1)
            \T^{r'}(v_4,v_1)\\
        &\quad\quad
            - 
            \T^{r}(u_2,v_3)
            \T^{r}(v_2,u_3) 
            \T^{r'}(u_4,v_1)
            \T^{r'}(v_4,u_1)
        \Big)
    d\vec v.
    \label{eq:clusterEqTwo}
\end{aligned}\end{equation}

Let us for now restrict ourselves to the case when
\begin{equation}\begin{aligned}
    r &\geq 2,\\
    r' &\geq r+2,
\end{aligned}\end{equation}
and note that this implies that $r+p_2+r',\ r'+p_1+r\geq 6$. We now rewrite Equation \eqref{eq:clusterEqTwo} using the approximation from Equation \eqref{eq:TapproximationsRgeqTwo} for each of the four terms in the inner bracket and the approximation from Equation \eqref{eq:TapproximationsRgeqFour} in the denominator. We get
\begin{equation}\begin{aligned}
    |g(\vec u)|
    &\leq
    \int_{[0,1]^4}\bigg(\frac{
            \T^{p_1-2}(v_1,v_2)
        }{
        c^2
        e^{-\frac{\beta}{2u_3}
           -\frac{\beta}{2u_4}}
        \Big(\int_{[0,1]^2}
            e^{-\frac{\beta}{2\tilde v}}
            \T^{r'+p_1+r-4}(\tilde v,\tilde v')
            e^{-\frac{\beta}{2\tilde v'}}
            d\tilde v d\tilde v'\Big)
        }\\
        &\quad\times 
        \frac{
            \T^{p_2-2}(v_3,v_4)
        }{
        c^2
        e^{-\frac{\beta}{2u_1}
           -\frac{\beta}{2u_2}}
        \Big(\int_{[0,1]^2}
            e^{-\frac{\beta}{2\tilde v}}
            \T^{r+p_2+r'-4}(\tilde v,\tilde v')
            e^{-\frac{\beta}{2\tilde v'}}
            d\tilde v d\tilde v'\Big)
        }
        \Big(\prod_{j=1}^4e^{-\frac{\beta}{2u_j}-\frac{\beta}{2v_j}}\Big)\\
            &\quad\times
        \int_{[0,1]^8} e^{-\sum_{j=1}^8\frac{\beta}{2z_j}}\Big|
            \T^{r-2}(z_1,z_2)
            \T^{r-2}(z_3,z_4)
            \T^{r'-2}(z_5,z_6)
            \T^{r'-2}(z_7,z_8)\\
        &\quad\quad
            - 
            \T^{r-2}(z_1,z_4)
            \T^{r-2}(z_3,z_2)
            \T^{r'-2}(z_5,z_8)
            \T^{r'-2}(z_7,z_6)
        \Big|dz_1\cdots dz_8\bigg)
    dv_1\cdots dv_4\\
    &=
    \prod_{j=1}^2 
    \bigg(\frac{
         \int_{[0,1]^2}
            e^{-\frac{\beta}{2v}}
            \T^{p_j-2}(v,v')
            e^{-\frac{\beta}{2v'}}
        dv dv'
    }{
         c^2\int_{[0,1]^2}
            e^{-\frac{\beta}{2\tilde v}}\T^{r+p_j+r'-4}(\tilde v,\tilde v')e^{-\frac{\beta}{2\tilde v'}}
            d\tilde vd\tilde v'
    }\bigg)\\   
    &\quad\times\int_{[0,1]^8}
            e^{-\sum_{j=1}^8\frac{\beta}{2z_j}}\Big|
            \T^{r-2}(z_1,z_2)
            \T^{r-2}(z_3,z_4)
            \T^{r'-2}(z_5,z_6)
            \T^{r'-2}(z_7,z_8)\\
        &\quad\quad\quad
            - 
            \T^{r-2}(z_1,z_4)
            \T^{r-2}(z_3,z_2)
            \T^{r'-2}(z_5,z_8)
            \T^{r'-2}(z_7,z_6)
        \Big|dz_1\cdots dz_8.
    \label{eq:clusterEqThree}
\end{aligned}\end{equation}
Note that all four variables $u_1,u_2,u_3$ and $u_4$ cancel in Equation \eqref{eq:clusterEqThree} and what remains is to bound an expression which only depends on $\beta$ and $\gamma$. Using the approximation in Equation \eqref{eq:splitTinTheMiddle} on the final integral gives
\begin{equation}\begin{aligned}
    &\int_{[0,1]^8}
    e^{-\sum_{j=1}^8\frac{\beta}{2z_j}}
    \Big|
            \T^{r-2}(z_1,z_2)
            \T^{r-2}(z_3,z_4)
            \T^{r'-2}(z_5,z_6)
            \T^{r'-2}(z_7,z_8)\\
        &\quad\quad\quad
            - 
            \T^{r-2}(z_1,z_4)
            \T^{r-2}(z_3,z_2)
            \T^{r'-2}(z_5,z_8)
            \T^{r'-2}(z_7,z_6)
    \Big|dz_1\cdots dz_8\\
    &=
    \int_{[0,1]^{12}}
        e^{-\sum_{j=1}^8\frac{\beta}{2z_j}
        -\sum_{j=1}^4\frac{\beta}{2v_j}}
    \Big|
            \T^{r-2}(z_1,z_2)
            \T^{r-2}(z_3,z_4)
            \T^{r-2}(z_5,v_1)
            \T^{r-2}(z_7,v_3)\\
        &\quad\quad\quad
            - 
            \T^{r-2}(z_1,z_4)
            \T^{r-2}(z_3,z_2)
            \T^{r-2}(z_5,v_3)
            \T^{r-2}(z_7,v_1)
    \Big|\\
    &\quad\quad\quad\times
    \T^{r'-r-2}(v_1,v_2)
    \T^{r'-r-2}(v_3,v_4)
    dv_1 \cdots dv_4 dz_1\cdots dz_8\\
    &\leq
    \bigg(\int_{[0,1]^{8}}
        e^{-\sum_{j=1}^8\frac{\beta}{2z_j}}
    |\Delta^{r-2}(\vec z)|d\vec z\bigg)\bigg(
    \int_{[0,1]^2}
            e^{-\frac{\beta}{2 v}}\T^{r'-r-2}( v, v')e^{-\frac{\beta}{2 v'}}
            d vd v'\bigg)^2,
            \label{eq:rPrimeSepFromDelta}
\end{aligned}\end{equation}
where the final step denotes by $\Delta^r(\vec z)$ the function
\begin{equation}\begin{aligned}
    \Delta^{r}(\vec z) \coloneqq{}&
        \T^{r}(z_1,z_2) 
        \T^{r}(z_3,z_4)
        \T^{r}(z_5,z_6)
        \T^{r}(z_7,z_8)\\
    {}&
        - 
        \T^{r}(z_1,z_4) 
        \T^{r}(z_3,z_2)
        \T^{r}(z_5,z_8)
        \T^{r}(z_7,z_6),
\end{aligned}\end{equation}
for $\vec z\in[0,1]^8$. We can now remove the dependence on $r'$ from Equation \eqref{eq:clusterEqThree} by using the bound in Equation \eqref{eq:rPrimeSepFromDelta} and apply the approximation in Equation \eqref{eq:splitTinTheMiddle-Denominator} on the denominator in Equation \eqref{eq:clusterEqThree}. We get that
\begin{equation}\begin{aligned}
    |g(\vec u)|
    &\leq
    \prod_{j=1}^2 
    \bigg(\frac{
         \int_{[0,1]^2}
            e^{-\frac{\beta}{2v}}
            \T^{p_j-2}(v,v')
            e^{-\frac{\beta}{2v'}}
        dv dv'
    }{
         c^4\Big(\int_{[0,1]^2}
            e^{-\frac{\beta}{2 v}}\T^{p_j-2}( v, v')e^{-\frac{\beta}{2 v'}}
            d vd v'\Big)
            \Big(\int_{[0,1]^2}
            e^{-\frac{\beta}{2 v}}\T^{2r-6}( v, v')e^{-\frac{\beta}{2 v'}}
            d vd v'\Big)
    }\\   
    &\quad\times
        \frac{\int_{[0,1]^2}
            e^{-\frac{\beta}{2 v}}\T^{r'-r-2}( v, v')e^{-\frac{\beta}{2 v'}}
            d vd v'}{
                \int_{[0,1]^2}
            e^{-\frac{\beta}{2 v}}\T^{r'-r-2}( v, v')e^{-\frac{\beta}{2 v'}}
            d vd v'
        }\bigg)\times
        \int_{[0,1]^8}
            e^{-\sum_{i=1}^8\frac{\beta}{2z_i}}|\Delta^{r-2}(\vec z)|
    d\vec z\\
    &=
    \frac{
         \int_{[0,1]^8}
            e^{-\sum_{i=1}^8\frac{\beta}{2z_i}}|\Delta^{r-2}(\vec z)|d\vec z
    }{
         c^8
            \Big(\int_{[0,1]^2}
            e^{-\frac{\beta}{2\tilde v}}\T^{2r-6}(\tilde v,\tilde v')e^{-\frac{\beta}{2\tilde v'}}
            d\tilde vd\tilde v'\Big)^2
    }
    ,\label{eq:clusterEqFour}
\end{aligned}\end{equation}
The function $\Delta^{r}(\vec z)$ is now bound by an extension of Lemma 3.1 in \cite{Turova_22}.
\begin{lemma}\cite{Turova_22}
    For any parameters $\beta,\gamma\in\real^+\setminus\{0\}$, any integer $r\in\{1,\dots,\lfloor (N-p_1-p_2)/2\rfloor\}$, and all vectors $\vec z\in[0,1]^8$ there is an $0<\epsilon<1$, only dependent on $\gamma$, such that
    \begin{equation}\begin{aligned}
            \int_{[0,1]^8} e^{-\sum_{i=1}^8\frac{\beta}{2z_i}}
            |\Delta^{r}(\vec{z})|d\vec z
        &\leq 
            \epsilon^{r+2}
        \left(\int_{[0,1]^{2}} e^{-\frac{\beta}{2u}}\T^{\lfloor r/2\rfloor}(u,v)e^{-\frac{\beta}{2v}}dudv\right)^{8},
        \label{eq:clusterLemmaOnDeltaFct}
    \end{aligned}\end{equation}
    and in the case when $\gamma=0$, one may set $\epsilon=0$.
    \label{lem:TurovaDeltaBound}
\end{lemma}
\begin{proof}
    If $\gamma=0$, the integrals factorise since they consist of products of the function $Q(\cdot,\cdot)$ defined in Equation \eqref{eq:Qdef}. The difference on the left hand side of Equation \eqref{eq:clusterLemmaOnDeltaFct}  is therefore identically zero and the theorem statement follows. We now consider the case when $\gamma>0$.
    
    Lemma 3.1 in \cite{Turova_22} gives the bound that for any $\gamma>0$, there exists an $0<\epsilon<1$ such that for any $\beta>0$, $\vec z\in[0,1]^4$ and $r\in\{1,\dots,\left\lfloor (N-2)/2\right\rfloor\}$, we have
    \begin{equation}\begin{aligned}
        |\T^{r}(z_1,z_2) 
        \T^{r}(z_3,z_4)
        - 
        \T^{r}(z_1,z_3) 
        \T^{r}(z_2,z_4)|
            &\leq
        \epsilon^{r+1} \Big( 
        \int_{[0,1]^2}
        2\T^{\lfloor r/2\rfloor}(u,v)
        e^{-\frac{\beta}{2v}}
        dudv
        \Big)^4.
\end{aligned}\end{equation}
    This corresponds to our case when $p_1=p_2=1$, but with only two factors, instead of four, in each term on the left-hand-side. The function in \cite{Turova_22} corresponding to our $\T^r$, call it $\T^r_{\scalebox{0.7}{\cite{Turova_22}}}$, is defined slightly differently and we have
    \begin{equation}
        \T^r_{\scalebox{0.7}{\cite{Turova_22}}}(x,y) = \T^{r-1}(x,y),
    \end{equation}
    in relation to the notation used here. The proof in \cite{Turova_22} is readily extended with some minor tweaks, so it will only be outlined here.
    
    A recurrence relation for $\Delta^r(\vec z)$ is obtained by making two integrals explicit on either side of each of the four factors. For $\vec z\in[0,1]^8$, we get
    \begin{equation}\begin{aligned}
        \Delta^{r}(\vec z) &=
            \T^{r}(z_1,z_2) 
            \T^{r}(z_3,z_4)
            \T^{r}(z_5,z_6)
            \T^{r}(z_7,z_8)\\
        &\quad\quad
            - 
            \T^{r}(z_1,z_4) 
            \T^{r}(z_3,z_2)
            \T^{r}(z_5,z_8)
            \T^{r}(z_7,z_6)\\
        &=  \int_{[0,1]^8} 
            Q(z_1,z_1')
            Q(z_3,z_3')
            Q(z_5,z_5')
            Q(z_7,z_7')\\
        &\quad\quad\quad\halfquad
            \times\bigg(
            \T^{r-2}(z'_1,z'_2)
            \T^{r-2}(z'_3,z'_4)
            \T^{r-2}(z'_5,z'_6)
            \T^{r-2}(z'_7,z'_8)\\
        &\quad\quad\quad\quad\halfquad
            - 
            \T^{r-2}(z'_1,z'_4)
            \T^{r-2}(z'_3,z'_2)
            \T^{r-2}(z'_5,z'_8)
            \T^{r-2}(z'_7,z'_6)
            \bigg)\\
        &\quad\quad\quad\halfquad\times
            Q(z'_2,z_2)
            Q(z'_4,z_4)
            Q(z'_6,z_6)
            Q(z'_8,z_8)
            d\vec z'\\
        &=    \int_{[0,1]^8} \Delta^{r-8}(\vec z')\prod_{i=1}^8 Q(z_i,z'_i)d\vec z'.
    \end{aligned}\end{equation}
    Doing this $j$ times, where
    \begin{equation}
        j=\left\lfloor\frac{r}{2}\right\rfloor
    \end{equation}
    gives
    \begin{equation}\begin{aligned}
        \Delta^{r}(\vec z^{0})
        &=    \int_{[0,1]^{8j}} 
        \Delta^{r-8j}(\vec{z}^{j})
        \prod_{l=1}^{j}
        \prod_{i=1}^8 Q(z_i^{l-1},z_i^{l})
        d\vec z^{1}\cdots d\vec{z}^{j}\\
        &=    \int_{[0,1]^{8}} 
        \Delta^{r-8j}(\vec{z}^{j})
        \prod_{i=1}^8 \T^{j-2}(z_i^{0},z_i^{j})
        d\vec z^{j}.
        \label{eq:DeltaRecurrenceOne}
    \end{aligned}\end{equation}
    The technique used in \cite{Turova_22} still applies, which exploits that the difference within the term $\Delta^{r-8j}(\vec{z}^{j})$ is small when at least one of the variables being integrated over is small. The adjustment needed from this point onwards in the proof is only some technical differences due to the recursion steps each introducing eight variables instead of four.

    A final remark is that for small $\gamma$, the proof involves writing $Q(x,y)$ as
    \begin{equation}
        Q(x,y)=
            e^{-\frac{\beta}{2x}-\frac{\gamma}{4x}}
            e^{-\frac{\beta}{2y}-\frac{\gamma}{4y}}
            -e^{-\frac{\beta}{2x}-\frac{\beta}{2y}-\frac{\gamma}{x+y}}\Big(
            1-e^{-\gamma\frac{(x-y)^2}{4xy(x+y)}}\Big)
    \end{equation}
    and then one chooses $\epsilon$ such that it satisfies
    \begin{equation}
        1-e^{-\gamma\frac{(x-y)^2}{4xy(x+y)}} < \epsilon.
    \end{equation}
    This confirms that we indeed have $\epsilon\to 0$ as $\gamma\to0$.
\end{proof}

Using Lemma \ref{lem:TurovaDeltaBound} the expression in Equation \eqref{eq:clusterEqFour} simplifies into
\begin{equation}\begin{aligned}
    |g(\vec u)|
    &\leq
    \frac{
         \int_{[0,1]^8}
            e^{-\sum_{i=1}^8\frac{\beta}{2z_i}}|\Delta^{r-2}(\vec z)|d\vec z
    }{
         c^8
            \Big(\int_{[0,1]^2}
            e^{-\frac{\beta}{2v}}\T^{2r-6}( v,v')e^{-\frac{\beta}{2v'}}
            dvdv'\Big)^2
    }\\
    &\leq
    \frac{
        \epsilon^{r}
        \left(\int_{[0,1]^{2}} e^{-\frac{\beta}{2u}}\T^{\lfloor \frac{r-2}{2}\rfloor}(u,v)e^{-\frac{\beta}{2v}}dudv\right)^{8}
    }{
         c^8
            \Big(\int_{[0,1]^2}
            e^{-\frac{\beta}{2v}}\T^{2r-6}( v,v')e^{-\frac{\beta}{2v'}}
            dvdv'\Big)^2
    }.
    \label{eq:clusterEqSix}
\end{aligned}\end{equation}
We can write
\begin{equation}
    2r-6 = \begin{cases}
        4\left\lfloor\frac{r-2}{2}\right\rfloor-2,\text{ if } r \text{ is even},\\
        4\left\lfloor\frac{r-2}{2}\right\rfloor,\text{ if } r \text{ is odd},
    \end{cases}
\end{equation}
and for the moment assume that $r$ is even and $r>26$. From Equation \eqref{eq:clusterEqSix}, we use the inequalities from Equations \eqref{eq:splitTinTheMiddle-Denominator} and \eqref{eq:splitTinTheMiddle} thrice and once, respectively, to get
\begin{equation}\begin{aligned}
    |g(\vec u)|
    &\leq
    \frac{
        \epsilon^{r}
        \left(\int_{[0,1]^{2}} e^{-\frac{\beta}{2u}}\T^{\lfloor \frac{r-2}{2}\rfloor}(u,v)e^{-\frac{\beta}{2v}}dudv\right)^{8}
    }{
         c^8
            \Big(\int_{[0,1]^2}
            e^{-\frac{\beta}{2v}}\T^{2r-6}( v,v')e^{-\frac{\beta}{2v'}}
            dvdv'\Big)^2
    }\\
    &\leq
    \frac{
        \epsilon^{r}
        \left(\int_{[0,1]^{2}} e^{-\frac{\beta}{2u}}\T^{\lfloor \frac{r-2}{2}\rfloor}(u,v)e^{-\frac{\beta}{2v}}dudv\right)^{8}
    }{
         c^{14}
            \Big(\int_{[0,1]^2}
            e^{-\frac{\beta}{2v}}\T^{\lfloor \frac{r-2}{2}\rfloor}( v,v')e^{-\frac{\beta}{2v'}}
            dvdv'\Big)^6
            \Big(\int_{[0,1]^2}
            e^{-\frac{\beta}{2v}}\T^{\lfloor \frac{r-2}{2}\rfloor-11}( v,v')e^{-\frac{\beta}{2v'}}
            dvdv'\Big)^2
    }\\
    &\leq
    \frac{
        \epsilon^{r}
        \left(\int_{[0,1]^{2}} e^{-\frac{\beta}{2u}}\T^{\lfloor \frac{r-2}{2}\rfloor-11}(u,v)e^{-\frac{\beta}{2v}}dudv\right)^{2}
        \left(\int_{[0,1]^{2}} e^{-\frac{\beta}{2u}}\T^{9}(u,v)e^{-\frac{\beta}{2v}}dudv\right)^{2}
    }{
         c^{14}
            \Big(\int_{[0,1]^2}
            e^{-\frac{\beta}{2v}}\T^{\lfloor \frac{r-2}{2}\rfloor-11}( v,v')e^{-\frac{\beta}{2v'}}
            dvdv'\Big)^2
    }\\
    &\leq
    \frac{
        \epsilon^{r}
    }{
         c^{14}
    }.
    \label{eq:clusterEqSeven}
\end{aligned}\end{equation}
For lower (and odd) values of $r$, the only change is the number of times you can apply the inequality in the denominator which results in a different exponent of $c$. The above is the worst case, so for all $r\geq 2$ we have that
\begin{equation}\begin{aligned}
    |g(\vec u)|
    &\leq
    \frac{
        \epsilon^{r}
    }{
         c^{14}
    }.
\end{aligned}\end{equation}
Letting
\begin{equation}\begin{aligned}
    \alpha &\coloneqq -\ln(\epsilon)\ (>0),\\
    C &\coloneqq c^{-14},
\end{aligned}\end{equation}
gives us, from Equations \eqref{eq:expressionToBound}, \eqref{eq:clusterEqOne} and \eqref{eq:clusterEqSeven}, that
\begin{equation}\begin{aligned}
    |f_{\vec Y_I|\vec Y_J=\vec y_J}(\vec y_I) - f_{\vec Y_I}(\vec y_I)|
    &\leq
    Ce^{-\alpha r}f_{\vec Y_I}(\vec y_I),
    \label{eq:defsOfAlphaandC}
\end{aligned}\end{equation}
which concludes the proof under the current assumptions that $p_1,p_2\geq 2$ and $r\geq 2$. For $p_1=1$, the same procedure works if one imposes that
\begin{equation}
    \T^{-1}(x,y)\coloneqq 1,
\end{equation}
with the additional constraint that the variables $x$ and $y$ are the same, so integrals over both get rewritten as a single integral. For smaller values of $r$ and $r'$, the small deviations from the above approach can be incorporated in the definition of the constant $C$.

The definition of $\alpha$ in Equation \eqref{eq:defsOfAlphaandC} together with the final remark in the proof of Lemma \ref{lem:TurovaDeltaBound} directly gives the final statement in the theorem that $\alpha\to\infty$ as $\gamma\to0$. This completes the proof.  
\hfill$\Box$

\subsection{Proof of the Central Limit Theorem \ref{thm:CLT}}
\label{section:CLT}
    The structure of our proof follows that of \cite{SS_2004_CLTforRFldWExpDecay}. However, 
    to use the results of \cite{SS_2004_CLTforRFldWExpDecay} in our setting and to optimise the rate of convergence in CLT, we need to introduce some 
    modifications. 
    
    Let $\vec X$ be a centralised version of $\vec Y$:
    \begin{equation}
        X_i \coloneqq Y_i-\E(Y_i),\ \ \ i\in\{1,\dots,N\}.
    \end{equation}
    The main idea is as follows. We partition the index set $\{1,\dots, N\}$ into alternating large and small blocks (which will be denoted $V$ and $W$, respectively), and consider the sums of the corresponding variables. Then we neglect the contribution of the smaller blocks in the total sum, while approximate the sums along the larger blocks by the independent sums. We begin with defining this partition in details.
    \begin{definition}
    Let $\delta_1,\delta_2\in(0,1)$ be real numbers satisfying $\delta_1+\delta_2<1$ and $\delta_1>\delta_2$. Define $p$, $q$ and $k$ to be integer functions of $N$ given by
    \begin{equation}\begin{aligned}
        p &\coloneqq \left\lfloor N^{\delta_1} \right\rfloor,\\
        q &\coloneqq \left\lfloor N^{\delta_2} \right\rfloor,\\
        k &\coloneqq \left\lfloor \frac{N}{p+q} \right\rfloor
            = N^{1-\delta_1} + \O(N^{1-2\delta_1+\delta_2}).
        \label{CLTeq:pqkDefs}
    \end{aligned}\end{equation}    
    Since we are interested in the limit as $N\to\infty$, we may restrict ourselves to the values of $N$ that additionally satisfy
    \begin{equation}
        N = k(p+q).
    \end{equation}
        
    Define a partition of the index set $\{1,\dots,N\}$ into $k$ subsets $(V_i)_{i=1}^k$ and $(W_i)_{i=1}^k$ of sizes $p$ and $q$, respectively, by letting
        \begin{equation}\begin{aligned}
            V_1 &\coloneqq \{1, \dots, p\},\\
            W_1 &\coloneqq \{p+1, \dots, p+q\},\\
            V_{i} &\coloneqq \{j+(i-1)(p+q): j\in V_1\},\\
            W_{i} &\coloneqq \{j+(i-1)(p+q): j\in W_1\}.
        \end{aligned}\end{equation}
        \label{def:index}
    \end{definition}
    With this definition we get
    \begin{equation}
        q\ll p \ll N,
    \end{equation}
    and
    \begin{equation}\begin{aligned}
        \{1,\dots,N\}=\bigcup_{i=1}^k (V_i\cup W_i)
    \end{aligned}\end{equation}
    as desired, with each set of $p$ consecutive indices being followed by a set of $q$ consecutive indices and vice versa.

    The next preparation needed for what follows is a closer consideration of the quantity $\sigma_N$, defined in \eqref{eq:sigma} as 
    \begin{equation}
        \sigma_N^2 \coloneqq \frac{1}{N} \text{Var}\bigg(\sum_{i=1}^N X_i\bigg).
        \label{eq:sigmaWithX}
    \end{equation}
    The result of 
    Corollary 2.4 in \cite{Turova_22} 
    establishes a convergence of the densities of $X_i=X_{i,N}$ (observe the dependence on $N$ of the distribution) as $N\rightarrow \infty$, which 
    yields that the limiting distribution has a positive variance, and also allows us to derive that
    \begin{equation}
        \frac{1}{N} \text{Var}\bigg(\sum_{i=1}^N X_i\bigg){\rightarrow }
        c
    \end{equation}
    as $N\rightarrow \infty$, for some constant $c\in\real$. This gives us the following relations
    \begin{equation}\begin{aligned}
        \sigma_N^2  &= \E\Big(\sum_{i=1}^N X_1X_i\Big)\\
            &=
        \E(X_1^2) + \sum_{i=1}^{\lfloor N/2\rfloor}\E(X_1X_i)
        +\sum_{i=\lfloor N/2\rfloor+1}^N\E(X_{1,N}X_{i,N})\\
        &\overset{N\to\infty}{\longrightarrow}
        c.
        \label{CLTeq:sigmaNonZero}
    \end{aligned}\end{equation}
    
    The positivity of $\sigma_N$ can further be used to find a lower bound on the variance of a partial sum. For $p\in\posintegers$ such that
    \begin{equation}\begin{aligned}
        p \overset{N\to\infty}{\longrightarrow}&{} \infty,\\
        p \ll\halfquad&{} N,
    \end{aligned}\end{equation}
    we have, by shifting indices and using Theorem \ref{thm:clusterDecayCoulomb} and Equation \eqref{CLTeq:sigmaNonZero}, that
    \begin{equation}\begin{aligned}
        \text{Var}&\Big(\sum_{j=1}^p X_{j,N}\Big)^2
            =
        \sum_{j_1=1}^{p}
        \Big(
            \E(X_{j_1,N}^2)
            +\sum_{j_2=j_1+1}^{p} \E(X_{j_1,N}X_{j_2,N})
            +\sum_{j_2=1}^{j_1-1} \E(X_{j_1,N}X_{j_2,N})
        \Big)\\
            &=
        2\sum_{j_1=1}^{\lfloor p/2 \rfloor}
        \Big(
             \E(X_{1,N}^2) 
             + \sum_{j_2=2}^{p-j_1+1} 
             \E(X_{1,N}X_{j_2,N})
             + \sum_{j_2=N-j_1+1}^{N} 
             \E(X_{j_2,N}X_{1,N})
        \Big)+o(1)\\
            &=
        \sum_{j_1=1}^{p}
        \Big(
             \E(X_{1,N}^2) 
             + \sum_{j_2=2}^{\lfloor N/2 \rfloor} 
             \E(X_{1,N}X_{j_2,N})
             + \sum_{j_2=\lfloor N/2 \rfloor+1}^{N} 
             \E(X_{j_2,N}X_{1,N})\\
             &\quad
        +o(e^{-\alpha(p-j_1+1)})
        +o(e^{-\alpha j_1})
        \Big)\\
            &=
        p\big(\sigma_N+o(1)\big).
    \end{aligned}\end{equation}
    This together with the non-zero limit of $\sigma_N$ in Equation \eqref{CLTeq:sigmaNonZero} means that for large $N$ (and hence large $p$), there exists a constant $c_1\in\real^+$ such that
    \begin{equation}
        \E\left(\sum_{i=1}^p X_{i,N} \right)^2
        \geq c_{1} p,
        \label{lemB-Eeq:Squarebound}
    \end{equation}
    
    We next make use of Theorem 2 in \cite{Petrov_75}, which states the following. Let $z\in\real$ and $G_1(z),G_2(z)\in\real$ be functions such that $G_1(z)$ is non-decreasing and $G_2(z)$ is differentiable with bounded variation. Suppose that $G_1(-\infty)=G_2(-\infty)$ and $G_1(\infty)=G_2(\infty)$. Let $g_1(t),g_2(t)$ be the Fourier-Stiltjes transforms of $G_1(z)$ and $G_2(z)$, respectively. Let $c'\in\real^+$ be a constant such that 
    \begin{equation}
        \sup_z|G_2'(z)|\leq c',
    \end{equation}
    and let $T$ be an arbitrary positive number. Then for every $b > (2\pi)^{-1}$ we have
    \begin{equation}
        \sup_z | G_1(z)-G_2(z) |
            \leq
        b\int_{-T}^T \Big| \frac{g_1(t)-g_2(t)}{t} \Big| dt + c(b)\frac{c'}{T},
        \label{eq:PetrovBound}
    \end{equation}
    where $c(b)$ is a positive constant which only depends on $b$. We let $b=1$, $G_1(z)$ be the probability distribution of the sum under consideration and $G_2(z)$ be the probability distribution of a standard normal $Z$. Then Equation \eqref{eq:PetrovBound} implies that there exists a constant $c_{2}\in\real^+$ such that
    \begin{equation}\begin{aligned}
        \sup_z\left| \prob\bigg(\frac{1}{\sqrt{\smash[b]{N\sigma^2_N}}}\sum_{j=1}^N X_j \leq z\bigg)-\prob(Z\leq z)\right| 
            &\leq
        \int_{-T}^{T} \frac{1}{|t|}\Big| \E\big(e^{it\frac{1}{\sqrt{\smash[b]{N\sigma^2_N}}}\sum_{j=1}^N X_j}\big)-e^{-t^2/2} \big| dt
        + 
        \frac{c_{2}}{T}.
        \label{CLTthmEq:PetrovResultOurVersion}
    \end{aligned}\end{equation}
    
    The difference in absolute value on the right hand side of Equation \eqref{CLTthmEq:PetrovResultOurVersion} we rewrite by addition and subtraction of intermediary terms, and then we bound the terms pairwise. For convenience, these bounds are stated as lemmas here and are proven (along the lines of \cite{SS_2004_CLTforRFldWExpDecay}) in sections \ref{sec:lem1_StoZeta} through \ref{sec:lem4_BerryEsseen} below. 
    
    Lemma \ref{lem:StoZeta} bounds the difference between considering the whole sum and only considering the terms from the larger index subsets $(V_i)_{i=1}^k$.
    \begin{lemma}\protect{\cite{SS_2004_CLTforRFldWExpDecay}}
        There exists a constant $c_{3}\in\real$ so that
        \begin{equation}\begin{aligned}
            \E\Big| \sum_{l=1}^k\sum_{j\in(V_l\cup W_l)}
            X_j
            -\sum_{l=1}^k\sum_{j\in V_{l}}X_{j} \Big| \leq 
        c_{3} \sqrt{kq}.
            \label{lemStoZetaeq:statement}
        \end{aligned}\end{equation}
        \label{lem:StoZeta}
    \end{lemma}
    This lemma implies that
    \begin{equation}\begin{aligned}
        \frac{1}{|t|}\bigg|
            \E\big(e^{it\frac{1}{\sqrt{\smash[b]{N\sigma^2_N}}}\sum_{j=1}^N X_j}\big)
        -
           \E\big(e^{it\frac{1}{\sqrt{\smash[b]{N\sigma^2_N}}}\sum_{l=1}^k\sum_{j\in V_{l}} X_{j}}\big)
        \bigg|
        \leq 
        \frac{c_{3}}{\sigma_N} \sqrt{\frac{kq}{N}}.
        \label{CLTthmEq:lemStoZeta}
    \end{aligned}\end{equation}

    The next lemma uses Theorem \ref{thm:clusterDecayCoulomb} to bound the error we get by neglecting the dependence between different blocks $V_i$, i.e. treating the sum over $V_{l_1}$ as independent from the sum over $V_{l_2}$ for all $l_1,l_2\in\{1,\dots,k\},\ l_1\neq l_2$. 
    \begin{lemma}
        There exists constants $c_{4},\alpha\in\real^+$ such that for all $t\in\real$, the following holds
        \begin{equation}\begin{aligned}
            \frac{1}{|t|}\Big|\E\Big(e^{\frac{it}{\sqrt{\smash[b]{N\sigma^2_N}}}\sum_{l=1}^k\sum_{j\in V_l}X_j }\Big) 
                  -\prod_{l=1}^k\E\Big(e^{\frac{it}{\sqrt{\smash[b]{N\sigma^2_N}}}\sum_{j\in V_1}X_{j} }\Big)\Big|
                \leq 
            \min\left\{c_{4}\frac{ke^{-\alpha q}}{|t|},2\sqrt{\frac{kp}{N}}\right\}.
            \label{CLTthmEq:lemZetatoXi}
        \end{aligned}\end{equation}
        \label{lem:ZetatoXi}
    \end{lemma}    

    Note that Equation \eqref{CLTthmEq:lemZetatoXi} contains a product of expected values. We shall define new random variables $\xi_i$ that are independent and identically distributed and consider them in what follows, letting
    \begin{equation}\begin{aligned}
        \xi_i \overset{d}{=} \sum_{j\in V_i}X_{j},\ \ \ i\in\{1,\dots,k\}.
        \label{eq:xiDistAsSum}
    \end{aligned}\end{equation}
    Lemma \ref{lem:XitoXi} exploits the independence of $\xi_i$ to rewrite the expression in preparation for Lemma \ref{lem:BerryEsseen}. 
    \begin{lemma}
        For any $\epsilon_5\in(0,1)$, there exists a positive constant $c_{5}\in\real$ such that for large $p$ we have
        \begin{equation}\begin{aligned}
            \E\Big|\sum_{i=1}^k \Big( \frac{\xi_i}{\sqrt{\smash[b]{N\sigma^2_N}}}-\frac{\xi_i}{\sqrt{k\text{Var}(\xi_i)}}\Big)\Big|
                \leq \frac{kq}{N} + c_{5} p^{-\epsilon_5}.
        \end{aligned}\end{equation}.
        \label{lem:XitoXi}
    \end{lemma}
    
    Lemma \ref{lem:XitoXi} implies that if $p$ is large enough, then for any $\epsilon_5\in(0,1)$ there exists a constant $c_{5}\in\real$ such that
    \begin{equation}\begin{aligned}
        \frac{1}{|t|}\Big|
            \E\Big(e^{\frac{it}{\sqrt{\smash[b]{N\sigma^2_N}}}\sum_{l=1}^k \xi_l}\Big)
        -
            \E\Big(e^{\frac{it}{\sqrt{k\text{Var}(\xi_1)}}\sum_{l=1}^k \xi_{l}}\Big)
        \Big|
            \leq 
        \frac{kq}{N} + c_{5} p^{-\epsilon_5}.
        \label{CLTthmEq:lemXitoXi}
    \end{aligned}\end{equation}

\begin{lemma}
    Let $C,\alpha\in\real^+$ be given by Theorem \ref{thm:clusterDecayCoulomb} for any $X_i,X_j$, $i,j\in\{1,\dots,p\}$. Define a constant $c_{6}\in\real^+$ by 
    \begin{equation}\begin{aligned}
        c_{6} &\coloneqq 
            \frac{16}{c_{1}^{3/2}}\Big(7+18\sum_{j=1}^\infty Ce^{-\alpha j}\Big),
        \label{lemB-Eeq:constant}
    \end{aligned}\end{equation}
    and let $T=T(N)\in\real^+$ be a function of $N$ that, for any $\epsilon_6\in(0,1)$ and $p$ large enough, satisfies
    \begin{equation}
        T \leq \frac{4\sqrt{kp^{1-\epsilon_6}}}{c_{6}}.
        \label{lemB-Eeq:Tbound}
    \end{equation} 
    Then for any $|t|\leq T$,
    we have
    \begin{equation}\begin{aligned}
        \Big|\E \Big(e^{ it\sum_{i=1}^k\frac{\xi_i}{\sqrt{\smash[b]{k\E(\xi_i^2)}}}}-e^{-t^2/2}\Big)\Big| &\leq
            \frac{c_{6}|t|^3e^{-t^2/3}}{ \sqrt{\smash[b]{kp^{1-\epsilon_6}}}}.
        \label{CLTthmEq:lemB-E}
    \end{aligned}\end{equation}
    \label{lem:BerryEsseen}
\end{lemma}
Now the results from Lemmas \ref{lem:StoZeta} through \ref{lem:BerryEsseen} are combined. For clarity, first a short summary of the constants and variables involved. We have $N\in\posintegers$ random variables $(X_i)_{i=1}^N$ and $k,p,q\in\posintegers$ are integer functions of $N$ given by Equation \eqref{CLTeq:pqkDefs}, such that $k(p+q)=N$. The index set is divided into subsets of consecutive indices $(V_i)_{i=1}^k$ and $(W_i)_{i=1}^k$ of sizes $p$ and $q$, respectively as in Definition \ref{def:index}. The random variables $(\xi_l)_{l=1}^k$ are independent and each distributed as the sum of $p$ consecutive $X_i$s. The cut-off $T\in\real$ will be defined to also be a function of $N$, while still satisfying the condition in Equation \eqref{lemB-Eeq:Tbound}. The factor $\sigma_N^2$ is given by Equation \eqref{eq:sigma} and the constants $C$ and $\alpha$ are given by Theorem \ref{thm:clusterDecayCoulomb}. The constants $c_{3},c_{4},c_{5},c_{6}$ are given by Lemmas \ref{lem:StoZeta} through \ref{lem:BerryEsseen} and $\epsilon_5,\epsilon_6\in(0,1)$ can be chosen arbitrarily.

Intermediary terms are added to the integral in Equation \eqref{CLTthmEq:PetrovResultOurVersion} so the bounds from Equations \eqref{CLTthmEq:lemStoZeta}, \eqref{CLTthmEq:lemZetatoXi}, \eqref{CLTthmEq:lemXitoXi} and \eqref{CLTthmEq:lemB-E} can be utilised. The constant factors in what follows are not strictly speaking relevant but are nevertheless included to aid the reader in following where each of the terms come from. We get
    \begin{equation}\begin{aligned}
        \sup_z\bigg| \prob\bigg(\frac{1}{\sqrt{\smash[b]{N\sigma^2_N}}}
        &\sum_{j=1}^N X_j \leq z\bigg)-\Phi(z)\bigg| 
    \leq
        \int_{-T}^{T} \bigg| \frac{\E\big(e^{it\frac{1}{\sqrt{\smash[b]{N\sigma^2_N}}}\sum_{j=1}^N X_j}\big)-e^{-t^2/2}}{t} \bigg| dt
        + 
        \frac{c_{2}}{T}\\
    &\leq
        \int_{-T}^{T} \frac{1}{|t|}\Bigg( \Big|
            \E\big(e^{it\frac{1}{\sqrt{\smash[b]{N\sigma^2_N}}}\sum_{j=1}^N X_j}\big)
            -
           \E\big(e^{it\frac{1}{\sqrt{\smash[b]{N\sigma^2_N}}}\sum_{l=1}^k\sum_{j\in V_{l}} X_{j}}\big)
        \Big|\\
      &\quad\quad+
        \Big|
            \E\Big(e^{\frac{it}{\sqrt{\smash[b]{N\sigma^2_N}}}\sum_{l=1}^k\sum_{j\in V_l}X_j }\Big) 
            -
            \prod_{l=1}^k\E\Big(e^{\frac{it}{\sqrt{\smash[b]{N\sigma^2_N}}}\sum_{j\in V_1}X_{j} }\Big)
        \Big|\\
      &\quad\quad+
        \Big|
            \E\Big(e^{\frac{it}{\sqrt{\smash[b]{N\sigma^2_N}}}\sum_{l=1}^k \xi_l}\Big)
            -
            \E\Big(e^{\frac{it}{\sqrt{k\text{Var}(\xi_1)}}\sum_{l=1}^k \xi_{l}}\Big)
        \Big|\\
      &\quad\quad+
            \Big|\E\Big(e^{\frac{it}{\sqrt{k\text{Var}(\xi_1)}}\sum_{l=1}^k \xi_l}-e^{-t^2/2}\Big)\Big|
        \Bigg)dt + \frac{c_{2}}{T}\\
    &\leq 
        \int_{-T}^{T} \bigg(
            \frac{c_{3}}{\sigma_N} \sqrt{\frac{kq}{N}}
            + 
            \Big(\frac{kq}{N} + c_{5} p^{-\epsilon_5}\Big)
            +
            c_{6} \frac{|t|^2e^{-t^2/3}}{\sqrt{kp^{1-\epsilon_6}}}
        \bigg)dt\\
        &\quad+
        \int_{-1/T}^{1/T} 
            \bigg(2\sqrt{\frac{kp}{N}}\bigg) dt
        +
        \int_{1/T \leq |t| \leq T} 
            \bigg(c_{4}\frac{ke^{-\alpha q}}{|t|}\bigg) dt
        + \frac{c_{2}}{T}\\
    &\leq 
        \Big(\frac{c_{3}}{\sigma_N}+1\Big) 2T \sqrt{\frac{kq}{N}}
          + 
            2c_{5} T p^{-\epsilon_5}
          +
            c_{6} \frac{5-3Te^{-T^2/3}}{\sqrt{kp^{1-\epsilon_6}}}\\
        &\quad+
            \Big(4\sqrt{\frac{kp}{N}}+c_{2}\Big)\frac{1}{T}
          +
            4c_{4} ke^{-\alpha q}\ln(T).
            \label{CLTthmEq:bigOne}
    \end{aligned}\end{equation}
    We now define the cut-off $T$ in terms of $N$ by letting $\delta_3\in(0,1)$ and defining
    \begin{equation}\begin{aligned}
        T &\coloneqq N^{\delta_3}.
        \label{CLTeq:Tdef}
    \end{aligned}\end{equation}
    The definitions of $p,q$ and $k$ in \eqref{CLTeq:pqkDefs} together with Equation \eqref{CLTeq:Tdef} are now used to write the non-exponential terms in Equation \eqref{CLTthmEq:bigOne} containing $p,q,k$ and $T$ as powers of $N$. 
    We get, for simplicity not including any constants, that
    \begin{equation}\begin{aligned} \begin{cases}
        \sqrt{\frac{kq}{N}} T
          &=
        \O(N^{-(\delta_1-\delta_2)/2+\delta_3}),\\
        T p^{-\epsilon_5}
            &= 
        \O(N^{\delta_3-\epsilon_5\delta_1}),\\
        \sqrt{k^{-1}p^{-1+\epsilon_6}}
            &=
        \O(N^{(-1+\epsilon_6\delta_1)/2}),\\
        T^{-1}&=\O(N^{-\delta_3}).
    \end{cases} \end{aligned}\end{equation}

    In order for all terms in Equation \eqref{CLTthmEq:bigOne} to be of order $N^{-1/4+\epsilon}$ for some given $\epsilon\in(0,1/4)$, we must therefore have the following: 
    \begin{equation}\begin{aligned}
        \begin{cases}
        1+2\delta_2+4\delta_3 
            &\leq 2\delta_1+4\epsilon,\\
        1+4\delta_3
            &\leq 4\epsilon_5\delta_1+4\epsilon,\\
        2\epsilon_6\delta_1
        &\leq 1+4\epsilon,\\
        1
        &\leq 4\delta_3+4\epsilon.
    \end{cases}
    \label{CLTeq:exponentCond}
    \end{aligned}\end{equation}
    Combining the first and last conditions in \eqref{CLTeq:exponentCond} gives 
    \begin{equation}
        1-4\epsilon \leq \delta_1-\delta_2,
    \end{equation}
    and since $0<\delta_2<\delta_1<1$, this confirms that $\epsilon\in(0,1/4)$, so $N^{-1/4+\epsilon}$ is indeed the best possible rate with the used method and the derived bounds. It is now straightforward to check that the choice 
    \begin{equation}\begin{aligned}
        \begin{cases}
    \delta_1 &= 1-2\epsilon,\\
    \delta_2 &= \epsilon,\\
    \delta_3 &= 1/4-\epsilon,\\
    \epsilon_5 &= (1+\epsilon)/2,\\ 
    \epsilon_6 &= (1-\epsilon)/2, 
    \end{cases}
    \end{aligned}\end{equation}
    i.e. letting the partition of the index set in definition \ref{def:index} be given by
    \begin{equation}\begin{aligned}
        p &\coloneqq \left\lfloor N^{1-2\epsilon} \right\rfloor,\\
        q &\coloneqq \left\lfloor N^{\epsilon} \right\rfloor,\\
        k &\coloneqq \left\lfloor \frac{N}{p+q} \right\rfloor,
        \label{CLTeq:partitionChoice}
    \end{aligned}\end{equation}
    and the cut-off point in the integration be
    \begin{equation}\begin{aligned}
        T &\coloneqq N^{1/4-\epsilon},
    \end{aligned}\end{equation}
    satisfies all of the conditions in Equation \eqref{CLTeq:exponentCond}. This, in turn, means that for large $N$, Equation \eqref{CLTthmEq:bigOne} becomes
    \begin{equation}\begin{aligned}
        \sup_z\left| \prob\bigg(\frac{1}{\sqrt{\smash[b]{N\sigma^2_N}}}\sum_{j=1}^N X_j \leq z\bigg)-\Phi(z)\right| 
    &=
        \O(N^{-\frac{1}{4}+\epsilon}),
    \end{aligned}\end{equation}
    and this concludes the proof of the theorem (given the proofs for the lemmas below).
    \hfill$\Box$
    
\subsection{Proof of Lemma \ref{lem:StoZeta}}
\label{sec:lem1_StoZeta}
    This proof follows \cite{SS_2004_CLTforRFldWExpDecay} and is included for completeness. We work with the expectation of the square of the left hand side of Equation \eqref{lemStoZetaeq:statement}. After subtracting the identical terms, the remaining terms have indices in the $k$ boxes that are of size $q$. Using Theorem \ref{thm:clusterDecayCoulomb} then gives constants $C$ and $\alpha$ such that
    \begin{equation}\begin{aligned}
        \E\Big( \sum_{i=1}^k\sum_{j\in(V_i\cup W_i)} X_j
        -\sum_{i=1}^k\sum_{j\in V_{i}}X_{j} \Big)^2
            &=
            \E\Big(\sum_{i=1}^k \sum_{j\in W_i} X_j\Big)^2\\
            &\leq
        kq
            \E \Big(X_{1}\sum_{j=1}^{N} X_j\Big)\\
            &\leq
        2kq    \sum_{r=0}^{\lceil N/2\rceil}Ce^{-\alpha r}.
        \label{lemStoZetaeq:sqBound}
    \end{aligned}\end{equation}
    We define the constant $c_{3}\in\real$ by
    \begin{equation}\begin{aligned}
        c_{3}^2 \coloneqq 2\sum_{r=0}^{\infty}Ce^{-\alpha r},
    \end{aligned}\end{equation}
    and by using Jensen's inequality in Equation \eqref{lemStoZetaeq:sqBound} we get \begin{equation}\begin{aligned}
        \E\Big| \sum_{i=1}^k\sum_{j\in(V_i\cup W_i)}
        X_j
        -\sum_{i=1}^k\sum_{j\in V_{i}}X_{j} \Big| \leq 
        c_{3} \sqrt{kq},
    \end{aligned}\end{equation} 
    and the lemma is proven.
    \hfill$\Box$
    
\subsection{Proof of Lemma \ref{lem:ZetatoXi}}
\label{sec:lem2_ZetatoXi}
    This proof closely follows \cite{SS_2004_CLTforRFldWExpDecay} but utilises Theorem \ref{thm:clusterDecayCoulomb}. Starting from the absolute value on left hand side of Equation \eqref{CLTthmEq:lemZetatoXi}, we add and subtract intermediary terms to extract the block $V_1$.
    Since the distribution of the sum over any single block is the same, this gives us
    \begin{equation}\begin{aligned}
        &\Big|\E\Big(e^{\frac{it}{\sqrt{\smash[b]{N\sigma^2_N}}}\sum_{l=1}^k\sum_{j\in V_l}X_j}\Big) 
            - \prod_{l=1}^k\E\Big(e^{\frac{it}{\sqrt{\smash[b]{N\sigma^2_N}}}\sum_{j\in V_1}X_j}\Big)\Big|\\
        &\quad\leq 
        \Big|\E\Big(
                e^{\frac{it}{\sqrt{\smash[b]{N\sigma^2_N}}}\sum_{j\in V_1}X_j}
                e^{\frac{it}{\sqrt{\smash[b]{N\sigma^2_N}}}
                    \sum_{l=2}^k\sum_{j\in V_l}X_j}
                \Big) 
            - \E\Big(e^{\frac{it}{\sqrt{\smash[b]{N\sigma^2_N}}}\sum_{j\in V_1}X_j}\Big)
              \E\Big(e^{\frac{it}{\sqrt{\smash[b]{N\sigma^2_N}}}
                \sum_{l=2}^k\sum_{j\in V_l}X_j}\Big)
        \Big|\\
            &\quad\quad 
        + \Big|\E\Big(e^{\frac{it}{\sqrt{\smash[b]{N\sigma^2_N}}}\sum_{j\in V_1}X_j}\Big)\Big|\cdot
            \Big|\E\Big(e^{\frac{it}{\sqrt{\smash[b]{N\sigma^2_N}}}
            \sum_{l=2}^k\sum_{j\in V_l}X_j}\Big)
        -\prod_{l=2}^{k}\E\Big(e^{\frac{it}{\sqrt{\smash[b]{N\sigma^2_N}}}\sum_{j\in V_1}X_j}\Big)\Big|.
        \label{lemZetatoXieq:expdiff}
    \end{aligned}\end{equation}
    Using Theorem \ref{thm:clusterDecayCoulomb}, we get constants $C$ and $\alpha$ that enables us to bound the first term in Equation \eqref{lemZetatoXieq:expdiff} by
    \begin{equation}\begin{aligned}
        &\Big|\E\Big(
                e^{\frac{it}{\sqrt{\smash[b]{N\sigma^2_N}}}\sum_{j\in V_1}X_j}
                e^{\frac{it}{\sqrt{\smash[b]{N\sigma^2_N}}}
                    \sum_{l=2}^k\sum_{j\in V_l}X_j}
                \Big) 
            - \E\Big(e^{\frac{it}{\sqrt{\smash[b]{N\sigma^2_N}}}\sum_{j\in V_1}X_j}\Big)
              \E\Big(e^{\frac{it}{\sqrt{\smash[b]{N\sigma^2_N}}}
                \sum_{l=2}^k\sum_{j\in V_l}X_j}\Big)
        \Big|\\
        &\quad\leq 
        Ce^{-\alpha q},
    \end{aligned}\end{equation}
    since the block $V_1$ and a block containing all indices in $V_2\cup\cdots\cup V_k$ are separated by at least $q$ indices. In the second term in Equation \eqref{lemZetatoXieq:expdiff}, the factorised exponent we simply bound by $1$. Repeating this process with all boxes $V_{a}$ for $a\in\{2,\dots,k\}$, we get that
    \begin{equation}\begin{aligned}
        &\Big|\E\Big(e^{\frac{it}{\sqrt{\smash[b]{N\sigma^2_N}}}\sum_{l=1}^k\sum_{j\in V_l}X_j}\Big) 
            - \prod_{l=1}^k\E\Big(e^{\frac{it}{\sqrt{\smash[b]{N\sigma^2_N}}}\sum_{j\in V_1}X_j}\Big)\Big|\\
        &\quad\leq 
        Ce^{-\alpha q} 
        +   \Big|\E\Big(e^{\frac{it}{\sqrt{\smash[b]{N\sigma^2_N}}}
            \sum_{l=2}^k\sum_{j\in V_l}X_j}\Big)
        -\prod_{l=2}^{k}\E\Big(e^{\frac{it}{\sqrt{\smash[b]{N\sigma^2_N}}}\sum_{j\in V_1}X_j}\Big)\Big|\\
        &\quad\leq Cke^{-\alpha q}.
    \end{aligned}\end{equation}
    Letting $c_{4}=C$, we get
    \begin{equation}\begin{aligned}
        \frac{1}{|t|}\Big|\E\Big(e^{\frac{it}{\sqrt{\smash[b]{N\sigma^2_N}}}\sum_{l=1}^k\sum_{j\in V_l}X_j}\Big) 
            - \prod_{l=1}^k\E\Big(e^{\frac{it}{\sqrt{\smash[b]{N\sigma^2_N}}}\sum_{j\in V_1}X_j}\Big)\Big|
        \leq 
        c_{4} \frac{ke^{-\alpha q}}{|t|},
        \label{lemZetatoXieq:first}
    \end{aligned}\end{equation}
    which is one of the two bounds in the statement of the lemma, Equation \eqref{CLTthmEq:lemZetatoXi}. For the other, we first bound the exponents separately. Note that
    \begin{equation}\begin{aligned}
        \E\Big( \frac{1}{\sqrt{\smash[b]{N\sigma^2_N}}}\sum_{l=1}^k\sum_{j\in V_l}X_j \Big)^2
            &=
        \frac{k}{N\sigma_N^2}\E\Big( \big(\sum_{j_1\in V_1}X_{j_1}\big)\big(\sum_{l=1}^k\sum_{j_2\in V_l}X_{j_2}\big)\Big)
            \\&\leq
        \frac{kp}{N\sigma_N^2} \E\Big(\sum_{j=1}^N X_1X_j\Big)
                        \\&=
        \frac{kp}{N},
        \label{lem4eq:expBound1}
    \end{aligned}\end{equation}
    and
    \begin{equation}\begin{aligned}
        \prod_{l=1}^k\E\Big( \frac{1}{\sqrt{\smash[b]{N\sigma^2_N}}}\sum_{j\in V_l} X_j \Big)^2
            &\leq
        \prod_{l=1}^k\frac{p}{N\sigma_N^2}\E\Big(\sum_{j=1}^p X_1X_j\Big)\\
            &\leq
        \prod_{l=1}^k\frac{kp}{N}\\
            &\leq
        \frac{kp}{N}.
        \label{lem4eq:expBound2}
    \end{aligned}\end{equation}
    Equations \eqref{lem4eq:expBound1} and \eqref{lem4eq:expBound2} give us
    \begin{equation}\begin{aligned}
        \E\Big| e^{\frac{it}{\sqrt{\smash[b]{N\sigma^2_N}}} \sum_{l=1}^k\sum_{j\in V_l}X_j}
         - e^{\frac{it}{\sqrt{\smash[b]{N\sigma^2_N}}} \sum_{j\in V_1}X_j} \Big|
    &\leq
        \E\Big| 
            \frac{t}{\sqrt{\smash[b]{N\sigma^2_N}}} \sum_{l=1}^k\sum_{j\in V_l}X_j
            -k\frac{t}{\sqrt{\smash[b]{N\sigma^2_N}}} \sum_{j\in V_1}X_j
        \Big|\\
    &\leq
        2\sqrt{\frac{kp}{N}}|t|,
    \end{aligned}\end{equation}
    implying that
        \begin{equation}\begin{aligned}
        \frac{1}{t}\Big|\E\Big(e^{\frac{it}{\sqrt{\smash[b]{N\sigma^2_N}}}\sum_{l=1}^k\sum_{j\in V_l}X_j}\Big) 
            - \prod_{l=1}^k\E\Big(e^{\frac{it}{\sqrt{\smash[b]{N\sigma^2_N}}}\sum_{j\in V_1}X_j}\Big)\Big|
        \leq 
        2\sqrt{\frac{kp}{N}},
        \label{lemZetatoXieq:second}
    \end{aligned}\end{equation}
    which is the second bound in Equation \eqref{CLTthmEq:lemZetatoXi}. Thus, the lemma is proved.
    \hfill$\Box$
    
\subsection{Proof of Lemma \ref{lem:XitoXi}}
\label{sec:lem3_XitoXi}
We work with the square of the expression of the statement. Since $\xi_i$ are identically distributed, we can write $\text{Var}(\xi_i) = \text{Var}(\xi_1)$ for all $i$. Independence means the sum can be extracted from the expectation and the expression rewritten as
    \begin{equation}\begin{aligned}
        \E\bigg(\sum_{i=1}^k \Big( \frac{\xi_i}{\sqrt{\smash[b]{N\sigma^2_N}}}
            -\frac{\xi_i}{\sqrt{k\text{Var}(\xi_i)}}\Big)\bigg)^2
        &= \bigg(\frac{1}{\sqrt{\smash[b]{N\sigma^2_N}}}
            -\frac{1}{\sqrt{k\text{Var}(\xi_1)}}\bigg)^2\sum_{i=1}^k \E(\xi_i^2) \\
        &= \bigg(\frac{1}{\sqrt{\smash[b]{N\sigma^2_N}}}
            -\frac{1}{\sqrt{k\text{Var}(\xi_1)}}\bigg)^2 k\text{Var}(\xi_1) \\
        &= \bigg(\sqrt{\frac{k\text{Var}(\xi_1)}{N\sigma_N^2}}-1\bigg)^2 \\
            &\leq \bigg(1-\frac{k\text{Var}(\xi_1)}{N\sigma_N^2}\bigg)^2 \\
            &\leq \bigg(1-\frac{pk}{N}\frac{\text{Var}(\xi_1)}{\sigma_N^2p}\bigg)^2 \\
            &\leq \bigg(\left|1-\frac{pk}{N}\right| +\left| 1-\frac{\text{Var}(\xi_1)}{\sigma_N^2p}\right|\bigg)^2\\
            &\leq \bigg(\frac{qk}{N} + \left| 1-\frac{\text{Var}(\xi_1)}{\sigma_N^2p}\right|\bigg)^2.
            \label{lemXitoXieq:DiffBound}
    \end{aligned}\end{equation}
    Now we bound the term in absolute value in the final bracket, for which we need additional steps compared with \cite{SS_2004_CLTforRFldWExpDecay}. For simplicity, we multiply by $\sigma_N^2 p$. We then insert the definition of $\sigma_N^2$ from Equation \eqref{eq:sigma}, use symmetry to shift the indices, and finally apply the bound of Theorem \ref{thm:clusterDecayCoulomb}. This results in the bound
    \begin{equation}\begin{aligned}
        p\sigma_N^2-\text{Var}(\xi_1)
            &= 
        p\sigma_N^2-\E\Big(\sum_{i=1}^p X_i\Big)^2\\
            &= 
        p\sum_{j=1}^N \E(X_1X_j)
                -\sum_{i,j=1}^p \E(X_iX_j) \\
            &\leq 
        \bigg(p\E(X_1^2)+2p\sum_{j=2}^{\lceil N/2\rceil} \E(X_1X_j)\bigg)
        -\bigg(p\E(X_1^2)+2\sum_{j=2}^{p}(p-j+1) \E(X_1X_{j})\bigg) \\
            &= 
        2p\sum_{j=p+1}^{\lceil N/2\rceil} \E(X_1X_j)
        +2\sum_{j=2}^p(j-1) \E(X_1X_j). 
        \label{eq:lemXiXiEq}
    \end{aligned}\end{equation}
    From the last line of Equation \eqref{eq:lemXiXiEq}, the first term is bound by
    \begin{equation}\begin{aligned} 
        2p\sum_{j=p+1}^{\lceil N/2\rceil} \E(X_1X_j)
            &\leq 
        2pe^{-\alpha p}\sum_{j=1}^{\infty} Ce^{-\alpha j}\\
            &\leq
        2c_{} p e^{-\alpha p},
        \label{lemXitoXieq:term1}
    \end{aligned}\end{equation}
    where 
    \begin{equation}\begin{aligned} 
        c\coloneqq \sum_{j=1}^{\infty} Ce^{-\alpha j}.
    \end{aligned}\end{equation}

    To get the desired bound on the final term in Equation \eqref{eq:lemXiXiEq}, we let $R$ be some index $R\in\{2,\dots,p\}$ and apply the bounds
    \begin{equation}
        \E(X_1X_{j}) \leq
        \begin{cases}
            1,\ &\text{ if } j\in\{2,\dots,R+1\},\\
            Ce^{-\alpha (j-1)}.\ &\text{ if } j\in\{R+2,\dots,p\}.
        \end{cases}
    \end{equation}
    This gives us
    \begin{equation}\begin{aligned}
        2\sum_{j=2}^p(j-1) \E(X_1X_j)
        &\leq 
            2\sum_{j=2}^R (j-1)
             +2\sum_{j=R+1}^{p} (j-1)Ce^{-\alpha (j-1)}\\
        &\leq
            R^2  + 2pe^{-\alpha R}\sum_{j=1}^{\infty} Ce^{-\alpha j}\\
        &\leq
            R^2  + 2c_{} pe^{-\alpha R}.
            \label{lemXitoXieq:term2}
    \end{aligned}\end{equation}
    We now define $R$ to be
    \begin{equation}\begin{aligned}
        R &\coloneqq\left\lfloor p^{(1-\epsilon_5)/2}\right\rfloor.
    \end{aligned}\end{equation}
    This choice together with Equations \eqref{lemXitoXieq:term1} and \eqref{lemXitoXieq:term2} inserted into Equation \eqref{eq:lemXiXiEq} gives
    \begin{equation}\begin{aligned}
        \left|1-\frac{\text{Var}(\xi_1)}{p\sigma_N^2}\right| 
            &\leq 
        \frac{1}{p\sigma_N^2}\Big(
            2c_{} p e^{-\alpha p}
            +R^2
            +2c_{} pe^{-\alpha R}
        \Big)\\
            &\leq
        \frac{1}{\sigma_N^2}\Big(
            4c_{} e^{-\alpha R}
            +p^{-\epsilon_5}\Big)\\
        &\leq 
        \frac{1}{\sigma_N^2}
            (4c_{}+1) p^{-\epsilon_5}\\ 
            &=
        c_{5} p^{-\epsilon_5},
        \label{lemXitoXieq:VarBound}
    \end{aligned}\end{equation}
    where
    \begin{equation}\begin{aligned}
        c_{5} 
            \coloneqq 
        \frac{1}{\sigma_N^2} + \frac{4}{\sigma_N^2}\sum_{j=1}^{\infty} Ce^{-\alpha j}.
    \end{aligned}\end{equation}
    Inserting this into Equation \eqref{lemXitoXieq:DiffBound} and using Jensen's inequality, we finally get
    \begin{equation}\begin{aligned}
            \E\bigg|\sum_{i=1}^k \Big( \frac{\xi_i}{\sqrt{\smash[b]{N\sigma^2_N}}}
                -\frac{\xi_i}{\sqrt{k\text{Var}(\xi_i)}}\Big)\bigg|
        &\leq 
            \frac{qk}{N} + \left| 1-\frac{\text{Var}(\xi_1)}{\sigma_N^2p}\right|\\
        &\leq 
            \frac{qk}{N} + c_{5} p^{-\epsilon_5},
        \end{aligned}\end{equation}
    and the Lemma is proved.
    \hfill$\Box$
    
\subsection{Proof of Lemma \ref{lem:BerryEsseen}}
\label{sec:lem4_BerryEsseen}
    Recall the Berry-Esseen Theorem stated in \cite{Petrov_75} as follows. 
    \begin{theorem*}[\cite{Petrov_75}, Theorem 2].
        Let $(\xi_i)_{i=1}^k$ be $k$ independent, zero mean random variables distributed as $\xi$ with bounded third moment. Then for all $t\in\real$ satisfying
        \begin{equation}\begin{aligned}
            |t|\leq \frac{(k\E(\xi^2))^{3/2}}{4k\E(|\xi|^3)},
            \label{lemB-Eeq:B-Econdition}
        \end{aligned}\end{equation}
        we have that
        \begin{equation}\begin{aligned}
            \left| \E\Big(e^{\frac{it}{\text{Var}(\xi_1)}\sum_{i=1}^k \xi_i}\Big)-e^{-t^2/2}\right| 
                \leq
            16|t|^3e^{-t^2/3}\frac{k\E(|\xi|^3)}{(k\E(\xi^2))^{3/2}}.
            \label{eq:BerryEsseen}
        \end{aligned}\end{equation}
    \end{theorem*}
    We will in what follows derive the bound
    \begin{equation}\begin{aligned}
        16\frac{k\E(|\xi|^3)}{(k\E(\xi^2))^{3/2}}
        \leq \frac{c_{6}}{\sqrt{kp^{1-\epsilon_6}}}.
        \label{lemB-Eeq:boundToShow}
    \end{aligned}\end{equation}
    If this indeed holds, then the bound for $T$ stated in the lemma means that also Equation \eqref{lemB-Eeq:B-Econdition} is satisfied for all $|t|\leq T$ since
    \begin{equation}\begin{aligned}
        T &\leq \frac{4\sqrt{kp^{1-\epsilon_6}}}{c_{6}}\\
        &\leq \frac{(k\E(\xi^2))^{3/2}}{4k\E(|\xi|^3)}.
    \end{aligned}\end{equation}
    Further, with the claimed bound in Equation \eqref{lemB-Eeq:boundToShow}, Equation \eqref{eq:BerryEsseen} becomes the lemma statement, Equation \eqref{CLTthmEq:lemB-E}, and the proof would be complete. 
    
    To prove that the bound in Equation \eqref{lemB-Eeq:boundToShow} holds, we first derive an upper bound for the third absolute moment of $\xi$. By collecting terms and shifting indices, we get the equality
    \begin{equation}\begin{aligned}
        \E(|\xi|^3) &= \E\Big(\Big|\sum_{j=1}^p X_j\Big|^3\Big)\\
        &= 
            p\E(|X_1|^3) 
            + 6\sum_{r=1}^{p-1}(p-r)\E(X_1^2|X_{1+r}|)\\
            &\quad
            + 6\sum_{r_1=1}^{p-2}\sum_{r_2=1}^{p-r_1-1} (p-r_1-r_2)\E(|X_1 X_{1+r_1} X_{1+r_1+r_2}|).
    \label{lemB-Eeq:CubeMomentBound}
    \end{aligned}\end{equation}
    The first term in Equation \eqref{lemB-Eeq:CubeMomentBound} is simply bound by 
    \begin{equation}\begin{aligned}
       p\E(|X_1|^3) 
        \leq 
       p\E(|X_1|) 
        \leq p.
        \label{lemB-Eeq:term1}
    \end{aligned}\end{equation}
    For the second term in Equation \eqref{lemB-Eeq:CubeMomentBound} we use Theorem \ref{thm:clusterDecayCoulomb}, which gives constants $C,\alpha\in\real^+$ and the bound
    \begin{equation}\begin{aligned} 
        \sum_{r=1}^{p-1}(p-r)\E(X_1^2|X_{1+r}|)
            &\leq
        p\sum_{r=1}^{\infty}Ce^{-\alpha r}
            = c_{} p,
        \label{lemB-Eeq:term2}
    \end{aligned}\end{equation}
    where 
    \begin{equation}\begin{aligned} 
        c_{} \coloneqq \sum_{r=1}^{\infty}Ce^{-\alpha r}.
        \label{CLTeq:cSumDef}
    \end{aligned}\end{equation}
    For the third term in Equation \eqref{lemB-Eeq:CubeMomentBound}, we introduce a cut-off at an index $R\in\{1,\dots,p\}$, soon defined as a function of $p$, and bound the terms in the sums by $1$ when $r_1,r_2\leq R$ and by Theorem \ref{thm:clusterDecayCoulomb} otherwise. We get
    \begin{equation}\begin{aligned}
            \sum_{r_1=1}^{p-2}&\sum_{r_2=1}^{p-r_1-1} (p-r_1-r_2)\E(|X_1 X_{1+r_1} X_{1+r_1+r_2}|)\\
        &\leq
            \sum_{r_1=1}^{R}
            \Big(
                \sum_{r_2=1}^{\min\{R,p-r_1-1\}} 
                    (p-r_1-r_2)
                +\sum_{r_2=R+1}^{p-r_1-1} 
                    (p-R)Ce^{-\alpha r_2}
            \Big)\\  
          &\quad 
            +\sum_{r_1=R+1}^{p-2}
                \sum_{r_2=1}^{p-r_1-1} 
                    (p-R)Ce^{-\alpha r_1}\\
        &\leq
            R^2(p-R-1)
                +R(p-R)e^{-\alpha R}\sum_{r_2=1}^{\infty} 
                    Ce^{-\alpha r_2}\\  
          &\quad 
            +(p-R)^2e^{-\alpha R} \sum_{r_1=1}^{\infty} 
                    Ce^{-\alpha r_1}\\
        &\leq
            R^2 p
            + 2c_{} p^2e^{-\alpha R},
        \label{lemB-Eeq:term3intermediary}
    \end{aligned}\end{equation}
    where $c$ is defined as above, Equation \eqref{CLTeq:cSumDef}. Now, letting
    \begin{equation}\begin{aligned}
        R &\coloneqq \left\lfloor p^{\epsilon_6/4}\right\rfloor,
    \end{aligned}\end{equation}
    then Equation \eqref{lemB-Eeq:term3intermediary} becomes
    \begin{equation}\begin{aligned}
        \sum_{r_1=1}^{p-2}\sum_{r_2=1}^{p-r_1-1} (p-r_1-r_2)\E(|X_1 X_{1+r_1} X_{1+r_1+r_2}|)
        &\leq
            (1+2c_{})p^{1+\epsilon_6/2},
        \label{lemB-Eeq:term3}
    \end{aligned}\end{equation}    
    for large $p$. Inserting the bounds from Equations \eqref{lemB-Eeq:term1}, \eqref{lemB-Eeq:term2} and \eqref{lemB-Eeq:term3} into Equation \eqref{lemB-Eeq:CubeMomentBound}, we get the bound
    \begin{equation}\begin{aligned}
        \E|\xi|^3 &\leq
            p
            + 6c_{} p
            + 6(1+2c_{})p^{1+\epsilon_6/2}\\
        &\leq
            (7+18c_{}) p^{1+\epsilon_6/2}.
    \end{aligned}\end{equation}
    This together with the lower bound on $\E(\xi^2)$ from Equation \eqref{lemB-Eeq:Squarebound} gives us that
    \begin{equation}\begin{aligned}
        \frac{16k\E(|\xi|^3)}{(k\E(\xi^2))^{3/2}}
            &\leq
        \frac{16(7+18c_{})kp^{1+\epsilon_6/2}}
        {(c_{1} kp)^{3/2}}\\
            &=
        \frac{c_{6}}{ \sqrt{k p^{1-\epsilon_6}}},
        \label{lemB-Eeq:Tneed}
    \end{aligned}\end{equation}
    where $c_{6}$ is defined as in the lemma statement, Equation \eqref{lemB-Eeq:constant}. Thus Equation \eqref{eq:BerryEsseen} is true and the Lemma is proved.    
    \vspace{-14pt}\hfill$\Box$

\section{Acknowledgement}
I would like to thank my supervisor Tatyana Turova for our many interesting discussions and her valuable insights and encouragements.

\begin{flushleft}

\end{flushleft}

\end{document}